\documentclass[draft]{amsart}
\usepackage{amssymb,latexsym,amsfonts,verbatim,amscd,ifthen}
\usepackage{color}

\numberwithin{equation}{section}

\theoremstyle{plain}

\newtheorem{theorem}[equation]{Theorem}   
\newtheorem{lemma}[equation]{Lemma}
\newtheorem{proposition}[equation]{Proposition}
\newtheorem{corollary}[equation]{Corollary}

\theoremstyle{definition}

\newtheorem{definition}[equation]{Definition}
\newtheorem{remark}[equation]{Remark}
\newtheorem{remarks}[equation]{Remarks}
\newtheorem{examples}[equation]{Examples}
\newtheorem{example}[equation]{Example}

\DeclareMathOperator{\Tor}{Tor} \DeclareMathOperator{\Hom}{Hom}
 
\DeclareMathOperator{\HH}{H}

\DeclareMathOperator{\id}{id}

\DeclareMathOperator{\rank}{rank}

\begin{document}

\renewcommand{\:}{\! :}
\newcommand{\p}{\mathfrak p}
\newcommand{\m}{\mathfrak m}
\newcommand{\e}{\epsilon}
\renewcommand\a{\alpha}
\renewcommand\b{\beta}
\newcommand\om{\omega}
\renewcommand\d{\delta}
\renewcommand\t{\tau}
\renewcommand\th{\theta}
\newcommand{\lra}{\longrightarrow}
\newcommand{\ra}{\rightarrow}
\newcommand{\rl}{\rangle}
\newcommand{\lal}{\langle}
\newcommand{\altref}[1]{{\upshape(\ref{#1})}}
\newcommand{\bfa}{\boldsymbol{\alpha}}
\newcommand{\bfb}{\boldsymbol{\beta}}
\newcommand{\bfg}{\boldsymbol{\gamma}}
\newcommand{\bfM}{\mathbf M}
\newcommand{\bfI}{\mathbf I}
\newcommand{\bfC}{\mathbf C}
\newcommand{\bfB}{\mathbf B}
\newcommand{\bfT}{\mathbf T}
\newcommand{\bfiT}{\mathbf \piT}
\newcommand{\bfi}{\mathbf i}
\newcommand{\bsfC}{\bold{\mathsf C}}
\newcommand{\bsfT}{\bold{\mathsf T}}
\newcommand{\smsm}{\smallsetminus}
\newcommand{\ol}{\overline}

\newlength{\wdtha}
\newlength{\wdthb}
\newlength{\wdthc}
\newlength{\wdthd}
\newcommand{\elabel}[1]
           {\label{#1}
            \settowidth{\wdthb}{\tt\small{#1}}
            \addtolength{\wdthb}{1ex}
            \smash{
             \raisebox{0.7\baselineskip}
                      {\color{red}
                       \hspace*{-\wdthb}\tt\small{#1}\hspace{1ex}}} }

\newcommand{\mlabel}[1]
           {\label{#1}
            \setlength{\wdtha}{\textwidth}
            \setlength{\wdthb}{\wdtha}
            \addtolength{\wdthb}{\marginparsep}
            \addtolength{\wdthb}{\marginparwidth}
            \setlength{\wdthc}{\marginparwidth}
            \setlength{\wdthd}{\marginparsep}
            \addtolength{\wdtha}{2\wdthc}
            \addtolength{\wdtha}{2\marginparsep}
            \setlength{\marginparwidth}{\wdtha}
            \setlength{\marginparsep}{-\wdthb}
            \setlength{\wdtha}{\wdthc}
            \addtolength{\wdtha}{1ex}
            \settowidth{\wdthb}{\tt\small{#1}}
            \addtolength{\wdthb}{1ex}
            \marginpar{\vspace*{\baselineskip}
            \smash{
               \raisebox{0.7\baselineskip}{\tt\small{#1}}
               \hspace{-\wdthb}
               \raisebox{.3\baselineskip}{\rule{\wdtha}{0.5pt}}} }
            \setlength{\marginparwidth}{\wdthc}
            \setlength{\marginparsep}{\wdthd}  }

\renewcommand{\mlabel}{\label}
\renewcommand{\elabel}{\label}


\title{Betti numbers of multigraded modules of generic type}
\author[H. Charalambous]{Hara Charalambous}
\thanks{The first author is grateful to  the European Union for support
in the framework of the program ``Pythagoras" of the ``Operational
Program for Education and Initial Vocational Training" of the 3rd
Community Support Framework of the Hellenic Ministry of Education, for part of the project.}
\address{Department of Mathematics\\
         Aristotle University of Thessaloniki\\
         Greece}
\email{hara@math.auth.gr}
\author[A. Tchernev]{Alexandre Tchernev}
\address{Department of Mathematics\\
         University at Albany, SUNY\\
         Albany, NY 12222}
\email{tchernev@math.albany.edu}

\begin{abstract}
Let $R=\Bbbk[x_1,\dots,x_m]$ be the polynomial ring over a field
$\Bbbk$ with the standard $\mathbb Z^m$-grading (multigrading),
let $L$ be a Noetherian multigraded $R$-module, let
$\beta_{i,\alpha}(L)$ the $i$th (multigraded) Betti number of $L$
of multidegree $\a$. We introduce the notion of a generic
(relative to $L$) multidegree, and the notion of multigraded
module of generic type. When the multidegree $\a$ is generic
(relative to $L$) we provide a Hochster-type formula for
$\beta_{i,\alpha}(L)$ as the dimension of the reduced homology of
a certain simplicial complex associated with $L$. This allows us
to show that there is precisely one homological degree $i\ge 1$ in
which $\beta_{i,\alpha}(L)$ is non-zero and in this homological
degree the Betti number is the $\beta$-invariant of a certain
minor of a matroid associated to $L$. In particular, this provides
a precise combinatorial description of all multigraded Betti
numbers of $L$ when it is a multigraded module of generic type.
\end{abstract}

\maketitle

\section*{Introduction}

Throughout this paper $\Bbbk$ is a field, $R=\Bbbk[x_1,\dots,x_m]$
is the polynomial ring over $\Bbbk$ with the standard $\mathbb
Z^m$-grading (multigrading), $L$ is a Noetherian multigraded
$R$-module with minimal multihomogeneous free presentation
$$
E \overset{\Phi}\lra G \lra L \lra 0,
$$
and $S$ is a (multi)homogeneous basis of $E$. An ongoing project
of the second author is to use the combinatorial properties of the
free multigraded resolution $T(\Phi, S)_\bullet$ of $L$ and the
matroid $\bfM(\Phi,S)$ from \cite{T} to study the homological
properties of $L$. In this paper we apply this technique to
investigate the multigraded Betti numbers $\beta_{i,\a}(L)$ in a
generic situation. We introduce the notion of a \emph{multigraded
module of generic type} which generalizes the notion of genericity
introduced previously by the authors in \cite{CT}. Our definition
is new even in the case of monomial ideals, where it properly (and
in a strong sense) subsumes the notion of generic ideals from
\cite{BPS}, and differs in an essential way from the notion of
genericity introduced in \cite{MSY}, see
Examples~\ref{E:generic-example}. We also introduce the finer
notion of a \emph{generic (relative to $L$) multidegree} (the
module $L$ is then of generic type if each multidegree
$\a\in\mathbb Z^m$ is generic relative to $L$).

When $I$ is a monomial ideal and $R/I$ is of generic type, as a
first result we show that the algebraic Scarf complex is a minimal
resolution of $R/I$, Corollary \ref{alg-Scarf-cor}.  Later on we
show that if $I$ is a monomial ideal and $\a$ is a generic
multidegree  then the $i$th Betti number of $R/I$ is nonzero
precisely when $\a$ belongs to the Scarf complex of $I$, Corollary
\ref{monomial-corollary}.

The main result of this paper is the computation of the $i$th
Betti number of $L$ when $\a$ is generic relative to $L$, Theorem
\ref{main-thm}. We provide a Hochster-type formula for the
multigraded Betti numbers $\beta_{i,\a}(L)$ in terms of the
relative homology of a certain simplicial complex associated with
$\a$ and the matroid $\bfM(\Phi,S)$. We analyze the properties of
that simplicial complex to show that the Betti numbers
$\beta_{i,\a}(L)$ are zero except in at most one degree $i$, and
in that degree the Betti number equals the $\beta$-invariant of a
certain minor $\bfM_\a$ of the matroid $\bfM(\Phi,S)$. In
particular, this provides a detailed combinatorial description of
the multigraded Betti numbers for the class of multigraded modules
of generic type.

The material is organized as follows. In Section~2 we introduce
the notion of a {\it generic element} in the LCM-lattice of $L$,
the notion of a {\it generic multidegree} relative to $L$, and the
notion of a multigraded module of {\it generic type}. We compare
the new notions with the existing notions of genericity. We show
that if $I$ is a  generic in the sense of \cite{BPS} monomial
ideal then the module $R/I$ is automatically of generic type. We
show that when $L$ is of generic type, then the proof of
\cite[Theorem 5.6]{CT} works and generalize the above theorem. In
particular this implies that the algebraic Scarf complex is a
minimal free resolution of $R/I$ when $R/I$ is a multigraded
module of generic type.

In Section~3 we introduce the affine simplicial complexes of a
matroid $\bfM$ on the set $S$. We show their homology equals the
$\beta$-invariant of the matroid. When a representation $\phi$ of
the matroid is given, then for any ordering $\om$ of $S$ we
introduce a certain complex of vector spaces $V(\phi,\om)$ with
the same homology.

In Section~4 when $\a$ is a generic element we examine certain
minors $\bfM_\a$ and $\bfM^\a$ of the matroid $\bfM(\Phi, S)$ of
$L$. We introduce  a new complex of vector spaces $V(\a,\phi,\om)$
whose homology is nonzero at precisely one position and equals the
$\beta$-invariant of  $\bfM_\a$.

Finally, in Section~5 we use the free multigraded resolution
$T_\bullet(\Phi,S)$ of \cite{T} to construct a certain double
complex. According to the first filtration of this complex we
recover the complex $V(\a,\phi,\om)$, while according to the
second filtration we can compute the $\a$-graded Betti numbers of
$L$. Thus we prove that when $\a$ is generic relative to $L$, the
reduced homology of the affine simplicial complexes of $\bfM_\a$
determines the $\a$-Betti numbers of $L$. Moreover the $\a$-Betti
number of $L$ is nonzero in exactly one position and equals the
$\beta$-invariant of a minor of $\bfM_\a$.

\section{Preliminaries}

For the rest of this paper  $L$ is a Noetherian multigraded
$R$-module, $\Phi\: E\lra G$ is a minimal free multigraded
presentation of $L$, and $S$ a homogeneous basis of $E$. All
vector spaces, homomorphisms and unadorned tensor operations are
over $\Bbbk$.

\subsection{Complexes of vector spaces}
Let $C_\bullet=(C_i,\varphi_i)$ be a complex of vector spaces. The
\emph{dual complex} of $C_\bullet$ is the complex
$$
C^*_\bullet=(C^*_i,\varphi^*_i)= \bigl(\Hom_\Bbbk(C_{-i},\Bbbk),
\Hom_\Bbbk(\varphi_{-i+1},\Bbbk)\bigr),
$$
its \emph{shift} by an integer $k$ is the complex
$$
C[k]_\bullet=(C[k]_i,\varphi[k]_i)=(C_{i+k},(-1)^k\varphi_{i+k}).
$$
its \emph{shift in homological degrees} by $k$ is the complex
$$
C\langle
k\rangle_\bullet=(C[k]_i,\varphi[k]_i)=(C_{i+k},\varphi_{i+k}),
$$
and its \emph{truncation} at $k$ is the complex
$$
C^{\ge k}_\bullet = (C^{\ge k}_i, \varphi^{\ge k}_i)
               = (\tau^k_iC_i, \tau^k_i\varphi_i) \text{ where }
\tau^k_i=
\begin{cases}
\id  &\text{ if } i\ge k; \\
0    &\text{ otherwise.}
\end{cases}
$$
We call $C_\bullet$ \emph{acyclic} if $\HH_i(C_\bullet)=0$ for
$i\ne 0$, and \emph{exact} if also $\HH_0(C_\bullet)=0$.

\subsection{Maps and vector spaces} We recall here and in the following subsections
 some of the notation introduced in
\cite{T}. Let $U_S$ be the vector space with basis the set of
symbols $\{e_a\mid a\in S\}$.  For each subset $A\subseteq S$ we
denote by $U_A$ the subspace of $U_S$ spanned by the set
$\{e_a\mid a\in A\}$. Whenever a map of vector spaces $\phi\: U_S
\lra W$ is given and $A\subseteq S$ we denote by $V_A$ the
subspace of $W$ spanned by the set $\{\phi(e_a) \mid a\in A\}$;
thus $V_A=\phi(U_A)$. We denote by $\phi_A\: U_A \lra W$ the
restriction of $\phi$ to $U_A$.

Next, we consider $\Bbbk$ as an $R$-module via the canonical
$\Bbbk$-algebra map $R\lra \Bbbk$ that sends each variable $x_i$
to $1\in\Bbbk$. We denote by $W(\Phi)$ the $\Bbbk$-vector space
$G\otimes_R\Bbbk$. Since the set $\{a\otimes 1 \mid a\in S\}$
forms a basis of the space $E\otimes_R\Bbbk$ we can canonically
identify it with $U_S$ by identifying $e_a$ with $a\otimes 1$ for
each $a\in S$. We denote by $\phi(\Phi)$ the $\Bbbk$-linear map
$\Phi\otimes_R\Bbbk$; thus we have $\phi(\Phi)\: U_S \lra
W(\Phi)$.

For  $A\subset S$ we let $E_A$ denote the free direct summand of
$E$ generated over $R$ by the set $\{a\mid a\in A\}$, and let
$\Phi_A$ denote the restriction of $\Phi$ to $E_A$. It is
straightforward that $\phi(\Phi_A)=\phi(\Phi)_A$.

\subsection{Matroids} For more details on the basic properties of
matroids we refer the reader to \cite{We, Ox}. For a quick summary
with notation in the spirit of this paper see \cite[Section~1]{T}.

Let $P(S)$ be the collection of all subsets of $S$, partially
ordered by inclusion. Recall that any matroid $\bfM$ on $S$ is
determined by a nonempty set $\mathcal{I}(\bfM)\subset P(S)$ (or
$\mathcal{I}$ if clear from the context) whose elements are called
the independent sets of $\bfM$. The set $\mathcal{I}$ has the
following three properties:
\begin{itemize}
\item{} $\emptyset \in \mathcal{I}$;
\item{} if $I\subset J$ and $J\in \mathcal{I}$ then $I\in \mathcal{I}$;
\item{} if $J\in P(S)$ and $I_1$, $I_2$ are two subsets of $J$ maximal with
        respect to membership in $\mathcal{I}$ then $|I_1|=|I_2|$. (This
        common size is called the rank of $J$ in $\bfM$ and is denoted
        by $r(J)$).
\end{itemize}

When a map of vector spaces $\phi\: U_S \lra W$ is given, one
obtains a matroid $\bfM(\phi)$ on $S$ by letting
the set $\mathcal{I}$ consist of all subsets $I$ of $S$ for which
$|I|=\dim_\Bbbk V_I$. In this case for any subset $J$ of $S$ we
let $r(J)=\dim_\Bbbk V_J$. One says that $\bfM(\phi)$ is
\emph{represented} by $\phi$ over $\Bbbk$ and that $\phi$ is a
representation over $\Bbbk$ of $\bfM(\phi)$. In the case at hand
where $\Phi\: E\lra G$ is a minimal free multigraded presentation
of $L$, we consider the matroid $\bfM(\Phi,
S):=\bfM\bigl(\phi(\Phi)\bigr)$.

Recall that a \emph{circuit} of a matroid $\bfM$ is a minimal
dependent set, and a \emph{loop} of $\bfM$ is an element $a\in S$
so that $\{a\}$ is a circuit. A \emph{T-flat} is a subset of $S$
that is a union of circuits. A \emph{flat} of $\bfM$ is subset
$B\subset S$ such that $r(B\cup c)=r(B) + 1$ for each $c\notin B$.
A \emph{hyperplane} of $\bfM$ is a maximal proper flat, i.e. a
flat $H$ such that $r(H)=r(S)-1$. The collection of T-flats forms
a lattice with respect to inclusion. The intersection of flats is
a flat. The \emph{matroid closure} $\ol B$ of a subset $B\subseteq
S$ is the smallest flat containing $B$: it equals the intersection
of all flats containing $B$. Two elements $a,b\in S$ are called
\emph{parallel} if they are not loops and $\ol{\{a\}}=\ol{\{b\}}$.
Let $J\subset S$. The \emph{restriction} of $\bfM$ to $J$ is the
matroid $\bfM |J$ whose independent sets form the set
$\mathcal{I}(\bfM|J)=\mathcal{I}(\bfM)\cap P(J)$. The
\emph{contraction} of $\bfM$ to $J$ is the matroid $\bfM.J$ whose
independent sets form the set $\mathcal{I}(\bfM .J)=\{ I\subset
J\mid I\cup I'\in \mathcal{I}(\bfM), \forall I'\in
\mathcal{I}(\bfM|S\setminus J)\}$. The $\beta$-invariant of $\bfM$
is
\[
\beta(\bfM)= (-1)^{r(S)} \sum_{J\subset S } (-1)^{|J|} r(J).
\]
One of the important properties of the $\beta$-invariant is that
$\beta(\bfM)=0$ is zero whenever $\bfM$ has a loop,
see \cite[Theorem II]{C}.

\subsection{Multigraded resolutions}\label{mult-res} For the rest of the paper
we denote by $T_\bullet(\Phi,S)$ the free multigraded resolution
of $L$ from \cite[Theorem 4.5]{T}. In homological degrees $0$ and
$1$ the resolution $T_\bullet(\Phi,S)$ is simply the minimal
presentation $\Phi$. For $n\ge 2$, the $R$-components of
$T_\bullet(\Phi,S)$ are:
$$
T_n(\Phi,S)=\bigoplus_I (T_I\otimes R)[-\deg I]
$$
where the index $I\subset S$ runs through all T-flats of $\bfM$
such that $r(I)=|I|-n+1$, $\deg I$ is the componentwise maximum of
the multidegrees of the elements of $I$, $T_I$ is a certain
$\Bbbk$-vector space associated to $I$, see \cite[Definition
2.2.3]{T}, and $(T_I\otimes R)[-\deg I]$ is the free module
$T_I\otimes R$ shifted by multidegree $\deg I$: $((T_I\otimes
R)[-\deg I])_\a= (T_I\otimes R)_{\a+\deg I}$ for any $a\in
\mathbb{Z}^m$.
 While in general the resolution
$T_\bullet(\Phi,S)$ is not minimal, we will use it to obtain
information about the minimal multigraded resolution of $L$. We
denote by $\beta_{i,\alpha}(L)$ the $i$th multigraded Betti number
of $L$:
$$
\beta_{i,\alpha}(L)= \dim_\Bbbk
\HH_i(T(\Phi,S)\otimes\Bbbk)_\alpha
                 = \dim_\Bbbk \Tor_i^R(L,\Bbbk)_\alpha.
$$

\section{Multigraded modules of generic type}

First we recall the definition of \emph{LCM-lattice} of the
multigraded module $L$.

\begin{definition}\mlabel{D:LCM-lattice}
Let $\Lambda=\Lambda(L)$ be the lattice in $\mathbb Z^m$ (with the
join operation being componentwise maximum) join-generated by the
multidegrees of the elements of $S$. We call $\Lambda$ the
\emph{LCM-lattice} of $L$. Since the collection of multidegrees
$\{\deg a\mid a\in S\}$ of the elements of $S$ is independent of
the choice of the basis $S$, the LCM-lattice is an invariant of
$L$.
\end{definition}

\begin{remark}\mlabel{R:LCM-degrees}
The multidegrees of the free modules in $T_\bullet(\Phi, S)$ are
elements of the LCM-lattice $\Lambda(L)$. It follows that the
minimal syzygies of $L$ can occur only in multidegrees $\alpha$
that belong to $\Lambda(L)$. Consequently, for $i\ge 1$ the Betti
numbers $\beta_{i,\alpha}(L)$ can be nonzero only if
$\alpha\in\Lambda(L)$.
\end{remark}

Consider the \emph{degree map} of posets
$$
\deg \: P(S) \lra  \mathbb Z^m
$$
given by $\deg A = \bigvee\{\deg a\mid a\in A\}$ for each subset
$A\subseteq S$, and note that $\Lambda(L)$ is precisely the image
of the map $\deg$. For each $\alpha\in\Lambda(L)$ there always is
a unique maximal set $I^\alpha$ in $P(S)$ of degree $\alpha$:
$I^\alpha$ equals the union of all sets of degree $\le\alpha$.

\begin{definition}\mlabel{D:Generic-type}
We say that $\alpha \in \Lambda(L)$ is a \emph{generic element of
$\Lambda(L)$} if the fiber $\deg^{-1}(\alpha)$ is a closed
interval in $P(S)$, i.e. if there is a unique minimal subset of
$S$ of degree $\alpha$ denoted by $I_{\alpha}$ and
$\deg^{-1}(\alpha)=[I_\alpha, I^\alpha]$.

We say that $\alpha\in\mathbb Z^m$ is \emph{generic relative to
$L$} if either $\alpha\notin\Lambda(L)$ or if $\alpha$ is a
generic element of $\Lambda(L)$.

We say that $\Lambda(L)$ is of \emph{generic type} if each
$\alpha\in\mathbb Z^n$ is generic relative to $L$ and in this case
we say that $L$ is also of \emph{generic type}.
\end{definition}

It is immediate that  if $L$ is of generic type then no two
elements of $S$ have the same multidegree. The notion of generic
type only depends on the multidegrees of the basis elements of $S$
and is independent of the choice of the particular basis $S$.
Below we give some examples to differentiate between the different
notions of generic.

\begin{examples}\mlabel{E:generic-example}
$ $
\begin{itemize}
\item{} The Scarf simplicial complex of $\Phi$, $\Delta(\Phi)$
is the subcomplex of $P(S)$ consisting of all subsets $I$ so that
$\deg^{-1} (\deg I)=\{ I\}$.
\item{} Let $J=(x^2, xy,xz)$. Then $R/J$ is of generic type as can be
readily checked, and $J$ is not generic in the sense of \cite{BPS}
or \cite{MSY}.

\item{} If $J$ is a monomial ideal  generic in the sense of
\cite{BPS},  then $R/J$ is of generic type. Indeed, for $\a\in
\Lambda(R/J)$ take $s_i\in S$ to be the unique monomial generator
of $I$ that agrees with $\a$ in the $i$th coordinate. The unique
minimal set $I_\alpha$ of degree $\alpha$ is the collection of the
distinct $s_i$ obtained this way.

\item{} Let $J=(x^3z^2, x^2y^3, xy^2z, y^3z^2)$. Then $(3,3,2)$ is not
a generic element of $\Lambda(R/J)$ as  $\{1,2\}, \{1,4\}$ are
minimal in $\deg^{-1}(3,3,2 )$. We note that $J$ is generic in the
sense of \cite{MSY}.

\item{} Let $I$ be a monomial ideal, $L=R/I$,  and $\Delta_L$ be the
Scarf complex of $L$, see \cite{BPS} or \cite{MSY}. If $\sigma\in
\Delta_L$ then the multidegree $\alpha$ of $\sigma$ is a generic
element of $\Lambda(L)$ and $\deg^{-1} (\alpha)=I^\a=I_\a$ is just
a point.

\item{} Let $I$ be a monomial ideal in $R$ and $J$ the polarization of
$I$ in a polynomial ring $S$. It is clear that $R/I$ is of generic
type if and only if $S/J$ is of generic type. This is not the case
when $I$ is generic in the sense of \cite{MSY} as the simple
example $(x^2, xy)$ demonstrates.

\item{} If $I$ is generic in the sense of \cite{MSY} then
$I^*=I+(x_1^D,\ldots, x_m^D)$ (where $D$ is sufficiently large) is
 also generic in the sense of \cite{MSY}. Let $I=(xy,
xz)$.  $R/I$ is of generic type and as we will see below  the
algebraic Scarf complex is a minimal free resolution of $R/I$.
However for $D>1$ the ideal $I^*=I+(x^D,y^D, z^D)$ is not of
generic type  since $\deg^{-1} ((1,D,D))$ is not an interval.

\end{itemize}
\end{examples}

Let $r=\rank \Phi$, $g=\rank_R(G)$. We recall from \cite{CT} that
$\Phi$ is
 of \emph{ uniform rank} if all $g\times r$ submatrices of the
 coefficient matrix of $\Phi$ have rank equal to $r$. In \cite{CT}
 the Scarf complex of $\Phi$, $S_\bullet(\Phi)$ was introduced.
 When $g=1$, $S_\bullet(\Phi)$ is the algebraic Scarf complex, $F_{\Delta(\Phi)}$, of
 \cite{BPS}. In \cite[Theorem 5.6]{CT} it is shown that when $\Phi$ is of uniform
 rank and $L$ is generic in the sense of \cite{BPS} then
$S_\bullet(\Phi)$ is a minimal free multigraded resolution of $L$.
The condition needed for the proof of \cite[Theorem 5.6]{CT} to
work is that there is a unique minimal face of degree $\a$,  so
that  $i\in I^\a\setminus I_\a$ if and only if $\deg
(I^\a\setminus i)=\deg (I^\a)=\a$, \cite[pg 547]{CT}. This
 condition
is is equivalent to $\a$ being a generic element of $\Lambda(L)$.
Thus the next theorem holds:

\begin{theorem} Let $\Phi: E\lra G$ be a minimal free multigraded presentation of
the multigraded module $L$ so that $\Phi$ is of uniform rank and
$L$ is of generic type. Then $S_\bullet (\Phi)$ is a minimal free
multigraded resolution of $L$.
\end{theorem}

We apply the above when $R/I$ is of generic type:

\begin{corollary}\label{alg-Scarf-cor} Let $I$ be a monomial ideal so that $R/I$ is of
generic type and $\Phi: R^n\lra R$ a minimal multigraded
presentation of $R/I$. The algebraic Scarf complex
$F_{\Delta(\Phi)}$ is a minimal free multigraded resolution of
$R/I$.
\end{corollary}

\begin{proof} Minimality of the presentation $\Phi$ implies that $\Phi$ is
  of uniform rank.
\end{proof}

Let $I$ be a monomial ideal and let $\Phi^*$ be the  minimal free
multigraded presentation of $R/I^*$. According to \cite{MSY}, $I$
is generic
 if and only if $F_{\Delta(\Phi^*)}$
is a minimal free resolution of $R/I^*$. Thus if $I$ is not
generic according to \cite{MSY} but $R/I$ is of generic type then
$F_{\Delta(\Phi^*)}$ is not acyclic and  $R/I^*$ is not  of
generic type.

\section
{ Affine simplicial complexes and the $\beta$-invariant}

Let $\bfM$ be an arbitrary matroid on the set $S$,  $\phi\: U_S
\lra W$  a representation of $\bfM$ and $\{e_a\mid a\in S\}$ a
basis of $U_S$. For any $b\in S$ we introduce a simplicial complex
and compute its homology.

\begin{definition}\mlabel{D:affine-simplicial-complex}
Let $b\in S$. We let $\Delta_b$  be the simplicial complex
\[ \Delta_b=\{ J\subset P(S)\mid b\notin\ol{J} \}\]
We call  $\Delta_b$ the \emph{affine simplicial complex of $\bfM$
away from $b$}.
\end{definition}

We note that the facets of $\Delta_b$ are the hyperplanes of the
matroid $\bfM$ that do not contain $b$ and that $J\in \Delta_b$ if
and only if $ V_J\ne V_{J\cup b}$.

\begin{theorem}\mlabel{homology-hyperplane-complex}
Let $b\in S$. Then for $i\ge 0$
\[
\dim_\Bbbk \widetilde{\HH}_{i}(\Delta_b,\Bbbk)=
\begin{cases}
\beta(\bfM)  &\text{ if } \ i = r(\bfM)-2 \\
0            &\text{ otherwise}.
\end{cases}
\]
Furthermore, if $b$ is not a loop then the above equality is true
for all $i$.
\end{theorem}

\begin{proof}
If $\bfM$ has a loop $c$ then $\beta(\bfM)=0$. Furthermore, in
that case $\Delta_b$ is either the empty simplicial complex (when
$b$ is a loop) or a cone with apex $c$. Thus in the sequel we
assume that $\bfM$ has no loops.

Next we define a new complex $\ol\Delta_b$ as follows: the set of
vertices of $\ol\Delta_b$ is
\[
Vert(\ol\Delta_b)= \{ J \mid J \textrm{ is a flat in $\bfM$ of
rank $1$ such that }
b\notin J 
\}.
\]
The faces of $\ol\Delta_b$ are the sets
\[
\{J_1,\ldots,J_t\}:\ J_1\cup\cdots \cup J_t\subset H,\ \textrm{for
some facet $H$ of \ } \Delta_b.
\]
Since $\bfM$ has no loops we can define the simplicial map
$\pi\:\Delta_b\lra\ol\Delta_b$ that sends each vertex $y$ of
$\Delta_b$ to its matroid closure $\ol{\{y\}}$, which is a flat of
rank $1$.
Since a facet of $\Delta_b$ contains $y$ if and only if it
contains $\ol{\{y\}}$, it is straightforward that $\pi$ is a
quotient map arising from partitioning the vertices of $\Delta_b$
into classes of parallel elements. Therefore the Contractible
Subcomplex Lemma \cite[(2.2)]{B2} yields that $\pi$ is a homotopy
equivalence.
Let $L_b$ be the poset obtained by removing from the lattice of
flats of $\bfM$ those flats that contain $b$.
Then by the Crosscut Theorem \cite[Theorem 2.3]{B}, $\ol\Delta_b$
is homotopy equivalent with the order complex of the poset $L_b^o$
obtained by removing from $L_b$ its minimal element. By the
results of Wachs and Walker \cite[Theorem 3.2 and Corollary
7.2]{WW} this order complex is a shellable simplicial complex, and
its reduced homology has already been computed, see e.g. \cite{B2}
and \cite[Theorem 2.6]{Z} or \cite[Theorem 3.12]{BZ}. In
particular, it is possibly nonzero only in dimension $r(\bfM) -
2$, and its rank there is precisely the $\beta$-invariant of
$\bfM$.
\end{proof}

Next we  compare the homology of $\Delta_b$ with the homology of a
complex of vector spaces determined by the subspaces of $V_S$. Let
$\om$ be an ordering on $S$ and we use this ordering to identify
each subset of $S$ with the increasing sequence of its elements.
For each $J\subset S$ and $c\notin J$, we have that $V_J$ is a
subset of $V_{J\cup\{c\}}$. We let
\[ V(\phi,\omega)_i=\bigoplus_{B\subset S, |B|=|S|-i} V_B\]
and $V(\phi,\omega)_\bullet$ be the complex

\begin{equation}\elabel{new-complex}
0\lra \bigoplus_{B\subset S, |B|=1} V_B  \lra\dots\lra
\bigoplus_{B\subset S, |B|=|S|-1} V_B \lra V_{S} \lra 0
\end{equation}
where at the $i$th stage the maps componentwise are the inclusions
$V_B\ra V_{B\cup\{c\}}$ times the sign of the permutation that
arranges the sequence $(c, B)$ in increasing order then followed
by composition with the natural inclusion
$V_{B\cup\{c\}}\hookrightarrow V(\phi,\omega)_{i+1}$.

\begin{lemma}\mlabel{homology-ommitted-complex}
If \ $b\in S$ is not a loop of \ $\bfM$ then
\[
\HH_i(V(\phi,\omega)_\bullet) \cong
\widetilde{\HH}_{|S|-i-2}(\Delta_b,\Bbbk)
\]
for all $i$.
\end{lemma}

\begin{proof} Let $\widetilde Y$ be the reduced chain complex of
$\Delta_b$ over $\Bbbk$. Since the complexes
$V(\phi,\omega)_\bullet$ are canonically isomorphic for different choices of
$\omega$, we may assume that $b$ is the greatest element of $S$
and that $\om$ induces the orientation on the faces of $\Delta_b$
used to construct $\widetilde Y$.  We  write $V_\bullet$ for
$V(\phi,\omega)_\bullet$. We consider a certain subcomplex
$D_\bullet$ of $V_\bullet$. We let $D_i=\bigoplus_J  V_J$ where
$b\in J$ and $|J|=|S|-i$. In particular $D_0=V_{S}$ and
\[
D_\bullet: 0\ra   V_{b}\lra\cdots\lra  V_S\ra 0
\]
where $D_i\ra D_{i-1}$ are the restrictions of the maps from
\altref{new-complex}.  Let
\[
E_\bullet=( V_\bullet/ D_\bullet)[1].
\]
If $b\notin J$, we let $\psi_J$ be the inclusion $ V_J \subset
V_{J\cup b}$ taken with the sign $(-1)^{|J|}$. Then $\psi\:
E_\bullet \lra D_\bullet$ defined componentwise by the maps
$\psi_J$ is  an injective  map of complexes and the complex $V_\bullet$
is precisely the mapping cone of $\psi$. We let
\[
\ol D_\bullet = D_\bullet/\psi(E_\bullet).
\]
It follows from the injectivity of $\psi$ and the standard properties
of mapping cones that
$\HH_i(V_\bullet) \cong \HH_i(\ol D_\bullet)$ for each $i$.
Next we note that the nonzero summands of $\ol
D_i$ are of the form $V_{J\cup b}/ V_J$ where $ V_J\ne V_{J\cup
b}$ and $|J|=|S|-1-i$. Thus the sets $J$ involved are
precisely the faces of $ \Delta_b$.
 We  consider
$Y'=\widetilde Y^*\lal-|S|+2\rl$, (so that that
$\Hom_\Bbbk(\widetilde Y_{-1}, \Bbbk)$ is in homological degree
$|S|-1$).  For $J\in \Delta_b$ we let $\sigma_J^*$ be the standard
generator of $Y'$ associated to  $J$.  Identifying $\sigma_J^*$
with $\phi(e_b)+V_J$ in $V_{J\cup b}/ V_J$  we see that
$Y^*\lal-|S|+2\rl$ can be identified with  $\ol D_\bullet$  and
the lemma follows.
\end{proof}

We are now ready to compute the homology of the complex
$V(\phi,\omega)_\bullet$:

\begin{theorem}\mlabel{homology-beta}
\[
\dim_\Bbbk\HH_{i}(V(\phi,\omega)_\bullet)=
\begin{cases}
\beta(\bfM)  &\text{ if } \ i = |S|-r(\bfM) \\
0            &\text{ otherwise}. 
\end{cases}
\]
\end{theorem}

\begin{proof}
If $\bfM$ has an element $b$ that is not a loop, then the theorem
is immediate by combining
Theorem~\ref{homology-hyperplane-complex} and
Lemma~\ref{homology-ommitted-complex}. If all elements of $S$ are
loops in $\bfM$ then $r(\bfM)=0$, the complex
$V(\phi,\omega)_\bullet$ is zero, and the desired conclusion is
immediate from the fact that $\beta(\bfM)=0$.
\end{proof}

\section{Minors of $\bfM(\Phi, S)$ associated with a generic multidegree}

Let $\a\in \Lambda(L)$ be a generic element. For the rest of this
paper we fix $\bfM=\bfM(\Phi,S)$, we set $W=W(\Phi)$, and
$\phi=\phi(\Phi)\: U_S \lra W$ (see Section~1). We
introduce some new matroids associated with $\a$. 
Recall that $\deg^{-1}(\alpha) = [I_\alpha , I^\alpha]$ is a
closed interval in the boolean poset $P(S)$.

\begin{definition}\mlabel{N:matroid-notation}
$ $
\begin{itemize}
\item{} We set $I(\alpha)=I^\alpha\setminus I_\alpha$.

\item{} We denote by $\bfM^\alpha$ the  matroid that is the
restriction of $\bfM$ to $I^\alpha$. In standard matroid notation 
we have $\bfM^\alpha=\bfM|I^\alpha$.

\item{}  We denote by $\bfM_\alpha$ the matroid that is the contraction of
$\bfM^\alpha$ to $I_\alpha$.  In standard matroid notation 
we have $\bfM_\alpha=\bfM^\alpha.I_\alpha =
\bfM^\alpha/I(\alpha)$.
\end{itemize}
\end{definition}

We discuss the above matroids in terms of some linear
transformations associated with $\alpha$.

\begin{remarks}\mlabel{R:minors-remarks}
$ $
\begin{itemize}
\item{} $\bfM^\alpha$ is represented by $\phi_{I^\alpha}: U_{I^\a}\ra W$
over $\Bbbk$.
\item{} Let $\pi_\a:
W\rightarrow W/V_{I(\a)}$ be the canonical projection map, and let
\[
\overline{\phi_{I_\a}}:= \pi_\a\circ  \phi_{I_\a} :\
U_{I_\a}\rightarrow W/V_{I(\a)}.
\]
The matroid $\bfM_\alpha$ is
represented by $\overline{\phi_{I_\alpha}}$ over $\Bbbk$.
\item{}
We set $\ol V_{I_\alpha}:=V_{I^\a}/V_{I(\a)}$. For each $B\subset
I_\a$, $V_{I(\a)}$  is a subspace of $V_{I(a)\cup B}$ and we set
$\ol V_{B}:= V_{I(a)\cup
B}/V_{I(\alpha)}=\overline{\phi_{I_\a}}(U_B)$.
\item{} Let $\om$ be an ordering on $I_\a$. According to the
definition
\begin{equation}\elabel{complex-bar-V}
V(\ol{\phi_{I_\alpha}},\omega)_\bullet:\quad 0\ra
\bigoplus_{B\subset I_\alpha, \ |B|=1 }
      \ol V_B
 \ra \dots
\ra
      \ol V_{I_\alpha} \ra 0.
\end{equation}
\end{itemize}
\end{remarks}

Next we define a new complex with the same homology as
$V(\ol{\phi_{I_\alpha}},\omega)_\bullet$.

\begin{definition}\mlabel{D:spaces-sequence}
Let $\om$ be an ordering on $I_\a$. We let
\[V(\a,\phi,\omega)_i=\bigoplus_{\begin{smallmatrix}
                     A\subset I_\alpha \\ |A|=i
                     \end{smallmatrix}} V_{I^\a\setminus A}\]
and define a complex of vector spaces $V(\a,\phi,\omega)_\bullet$
as the sequence

\begin{equation}\elabel{E:spaces-sequence}
0 \ra V_{I(\alpha)} \ra \dots \ra
      \bigoplus_{\begin{smallmatrix}
                     A\subset I_\alpha \\ |A|=1
                     \end{smallmatrix}}
      V_{I^\alpha\setminus A}
\ra
      V_{I^\alpha} \ra 0
\end{equation}
with maps that are componentwise the inclusions
$V_{I^\alpha\setminus (A\cup\{c\})}\subseteq V_{I^\alpha\setminus
A}$ times the sign of the permutation that arranges the sequence
$(c, I_\alpha\setminus A)$ in increasing order then followed by
composition with the natural inclusion.
\end{definition}

Let $\widetilde C$ be the reduced chain complex over $\Bbbk$ for
the full simplex on the set $I_\alpha$ as oriented by $\om$. Let
$C'=(\widetilde C)^*\lal-|I_\alpha|+1\rl$, (so that
$\Hom_\Bbbk(\widetilde C_{-1}, \Bbbk)$ is in homological degree
$|I_\alpha|$). We will consider the complex $V_{I(\alpha)}\otimes
C'$. Let $\sigma_A^*$ be the standard generator of $C'$ associated
to $A\subset I_\alpha$. We identify $V_{I(\alpha)}\otimes
\sigma_A^*$ with the subspace $V_{I(\alpha)}$ of $V_{I(\a)\cup B}$
via the map $v\otimes\sigma_A^* \mapsto v$, where
$B=I_\alpha\setminus A$. With this identification it is easy to
see that the following holds:

\begin{proposition}
The complex $V_{I(\a)}\otimes C'$ is a subcomplex of
$V(\alpha,\phi,\omega)_\bullet$ and we have the short exact
sequence of complexes
\begin{equation}\elabel{E:spaces-sequence-3}
0 \ra V_{I(\alpha)}\otimes C' \ra V(\alpha,\phi,\omega)_\bullet
  \ra V(\ol{\phi_{I_\alpha}},\omega)_\bullet \ra 0.
\end{equation}
\end{proposition}

We note that  $V_{I(\alpha)}\otimes C'$ is an exact complex.
Thus taking the long exact sequence on the homology of
\altref{E:spaces-sequence-3} yields

\begin{lemma}\mlabel{homology-bar} 
$ \HH_i(V(\alpha,\phi,\omega)_\bullet)=
\HH_i(V(\ol{\phi_{I_\alpha}}, \omega)_\bullet)$.
\end{lemma}

Combining Theorem~\ref{homology-beta} and Lemma~\ref{homology-bar}
we obtain the following

\begin{corollary}\mlabel{beta}
Let $\om$ be an ordering on $I_\a$. Then for all $i$ we have
\[
\dim_\Bbbk\HH_{i}(V(\a,\phi,\omega)_\bullet)=
\begin{cases}
\beta(\bfM_\a)  &\text{ if } \ i = |I_\a|-r(\bfM_\a)\\
0               &\text{ otherwise}.
\end{cases}
\]
\end{corollary}

We finish this section with an example to demonstrate the above.

\begin{example}\mlabel{E:minors-example}
Let $R=\mathbb{Q}[x,y,z]$, $E\cong R^4$ a multigraded free module
with basis $S=\{a,b,c,d\}$ where  $\deg a=(3,1,1)$, $\deg
b=(1,3,1)$, $\deg c=(1,1,3)$, $\deg d=(1,2,2)$, $G\cong R^2$ and
$L$ multigraded with minimal multigraded free presentation
$$ E
\overset{\Phi}\lra G \lra L \lra 0,
$$
and such that the matrix of $\phi(\Phi)$ according to the bases
$\{ e\otimes 1\mid e\in S\}$ and the canonical basis of $G$ is
given by the matrix
\[
\begin{bmatrix}
1&1&1&1\\
1&1&2&3
\end{bmatrix}.
\]
First we examine the case where $\a=(3,3,3)$. Here
$I^\a=\{a,b,c,d\}$, $I_\a=\{a,b,c\}$ and $I(\a)=\{d\}$. Thus
$\beta(\bfM_\a)= 1$, $r(M_a)=1$  and for any ordering $\om$ on
$I_\a$ the homology of the complex
\[
V(\a,\phi,\om):\quad 0\ra \mathbb{Q}\lra \mathbb{Q}^6\lra
\mathbb{Q}^6\lra \mathbb{Q}^2\ra 0
\]
is nonzero precisely for $i=2$.

When $\a=(3,2,3)$, $\deg^{-1}(\a)$ equals the point $\{a,c,d\}$,
while $r(M_a)=2$ and $\beta(M_\a)=1$. For any ordering $\om$ on
$I_\a$ the homology of the complex
\[
V(\a,\phi,\om):\quad 0\ra \mathbb{Q}^3\lra \mathbb{Q}^6\lra
\mathbb{Q}^2\ra 0
\]
is nonzero precisely for $i=1$.

\end{example}

\section{The Betti numbers of  $L$ }

Let  $T_\bullet(\Phi,S)$ be the multigraded free resolution of
$L$, see Section \ref{mult-res} and let $\a\in \Lambda(L)$. We
examine the $\a$-graded piece of $T_n(\Phi,S)$ for $n\ge 2$ in
order to compute $\beta_{i,\alpha}(L)$. We have that
\[
T_n(\Phi,S)_\a=\bigoplus_{J, \deg J+\b=\a} (T_J\otimes R)[-\deg
J]_\b
\]
where the index $J\subset S$ runs through all T-flats of $\bfM$
such that $r(J)=|J|-n+1$. It follows that $\deg J\le \a$ and that
$J\subset I^\a$. Let $m$ be the maximal multigraded ideal of $R$.
It is clear that
\[
mT_i(\Phi,S) \cap T_i(\Phi,S)_\alpha =
\bigoplus_{\begin{smallmatrix}
                    J ,\ \deg
J+\b=\a \\ \deg J< \a
                     \end{smallmatrix}} (T_J\otimes R)[-\deg J]_\b
                     \ . \]
When $\a$ is a generic element the condition $\deg J<\a $ is
equivalent to the existence of  an element $b\in I_\alpha$ such
that $J\subset I^\alpha\setminus \{b\}$.

Let $A\subset S$. We let
\[ \bfT_\bullet(\phi_A)=T_\bullet(\Phi_A,A)^{\ge 1}\langle
1\rangle\otimes_R\Bbbk \ .\] The complex $\bfT_\bullet(\phi_A)$
 was introduced in \cite[Definition 2.4.1]{T} where it was shown that
 \[ \bfT_\bullet(\phi_A)\ra V_A\ra 0\]
 is exact.  We  note that $\bfT_0(\phi_A)= U_A$. We will need the
 following important property, see \cite[Theorem~3.2(b) and
 Theorem~3.5]{T}: if $A\subset B$ then $\bfT_\bullet(\phi_A)$ is
canonically a subcomplex of
 $\bfT_\bullet(\phi_B)$. The following Lemma is
a straightforward consequence of the basic structure of these complexes.

\begin{lemma}\label{T-complexes} Let $\om$ be an ordering in $Y\subset S$ and $X_i$ a
        collection of subsets of $Y$.  There is a chain
        map $p: \bigoplus \bfT(\phi_{X_i})\rightarrow \bfT(\phi_Y)$ where
        $p|_{\bfT(\phi_{X_i})}$ equals the canonical inclusion map
        times the sign determined by $\om$ to order the elements
        of
        $(Y\setminus X_i,Y) $.
\end{lemma}

Let $\a\in \Lambda(L)$ be a generic element. Fix an
ordering $\omega$ in $I^\a$ and consider the chain map
\[
p: \bigoplus_{b\in I_\alpha}\bfT(\phi_{I^\alpha\setminus \{b\}})
 \ra \bfT(\phi_{I^\alpha})
\]
as in Lemma~\ref{T-complexes}. We introduce a new complex:

\begin{definition} Let $C(\alpha)_\bullet$ be the following
complex of vector spaces:
\[
C(\alpha)_\bullet \ = \ \bfT(\phi_{I^\alpha}) \ \bigg/ \
          p( \bigoplus_{b\in I_\alpha}
     \bfT(\phi_{I^\alpha\setminus \{b\}}).
\]
\end{definition}
We note that for $i\ge 1$ we have
$$
mT_i(\Phi,S) \cap T_i(\Phi,S)_\alpha=
p\bigl(\bigoplus_{b\in I_\alpha}\bfT(\phi_{I^\alpha\setminus \{b\}})\bigr)_{i-1}
$$
and thus the following
lemma holds:

\begin{lemma}\mlabel{first-quotient}
Let $\alpha\in\Lambda$ be a generic element. Then for $i\ge 1$
\[\beta_{i,\alpha}(L)=
\dim_\Bbbk\HH_{i-1}\bigl(C(\alpha)_\bullet\bigr).\]
\end{lemma}

\begin{proof}
Since $\beta_{i,\alpha}(L)=H_i(T(\Phi,S)\otimes_R \Bbbk)_\alpha, $
we combine the above remarks to get
\[
(T_i(\Phi,S)\otimes_R \Bbbk)_\alpha=T_i(\Phi,S)_\alpha/
mT_i(\Phi,S) \cap T_i(\Phi,S)_\alpha = C_{i-1}(\alpha).
\]
In view of the minimality of the
presentation $\Phi$, the Lemma is now immediate.
\end{proof}

We will reduce the study of the homology of $C(\alpha)_\bullet$ to
the study of a certain double complex. We will need the following
lemma:

\begin{lemma}\mlabel{T:bexactness} Let $X\subset
Y\subset S$, and $\omega$ be a linear ordering $Y$. Consider the
sequence
\begin{equation}\elabel{E:components-sequence-2}
\bfT(X,Y,\omega): 0\ra \bfT(\phi_X) \lra \bigoplus_{b\in
Y\setminus X}   \bfT(\phi_{X\cup\{b\}})\lra \bigoplus_{b,c\in
Y\setminus X} \bfT(\phi_{X\cup\{b,c\}})
\end{equation}
\[
\lra \dots\lra \bigoplus_{c \in Y\setminus X}\bfT
(\phi_{Y\setminus c})\ra
                          \bfT (\phi_Y) \ra 0
\]
where the morphism component $\bfT(\phi_{X\cup C})\lra
\bfT(\phi_{X\cup B})$ is \ $0$ \ if $C\not\subset B$ and otherwise is
the canonical inclusion times the sign determined by $\om$. Then
$\bfT(X,Y,\omega)$ is an acyclic complex.
\end{lemma}

\begin{proof}
We will do induction on $|Y\setminus X|$.  If $|Y\setminus X|=0$
then $Y=X$ and $\bfT(X,Y,\omega): 0\lra \bfT(\phi_Y)\lra 0$.
Suppose now that $X\neq Y$ and let $b$ be the biggest element of
$Y\setminus X$. We set $X'=X\cup \{b\}$, $Y'=Y\setminus \{b\}$ and
let $\omega'$ be the induced ordering on $Y'$. Then we get the
short exact sequence
\[
0\lra \bfT(X', Y, \omega')\lra \bfT(X,Y,\omega)
 \lra \bfT(X,Y',\omega')\lal-1\rl \lra 0.
\]
Using the induction hypothesis and the long exact sequence in
homology we get that if $i>1$ then $\HH_i( \bfT(X,Y,\omega) )=0$
while if $i=1$ then
\[
0\lra \HH_1( \bfT(X,Y,\omega) ) \lra \HH_0( \bfT(X,Y',\omega') )
 \lra \HH_0( \bfT(X',Y,\omega).
\]
Thus it suffices to show that the map of complexes
\[
\HH_0(\bfT(X,Y',\omega') )\lra \HH_0( \bfT(X',Y,\omega))
\]
induced by the inclusion map $\bfT(\phi_{Y'})\subset \bfT(\phi_Y)$
is injective. This however is immediate since
\[
\bfT_i(\phi_{Y'})\ \bigcap\ \bigoplus_{c\in Y\setminus X}
\bfT_i(\phi_{Y\setminus c})= \bigoplus_{c\in Y'\setminus X}
\bfT_i(\phi_{Y'\setminus c}),
\]
as follows from the structure of these sets, see \cite[Definition
2.4.1]{T}.
\end{proof}

We apply  Lemma~\ref{T:bexactness} to the special case where
$X=I(\alpha)$ and $Y=I^\alpha$. We have that $Y\setminus
X=I_\alpha$.

\begin{lemma}\mlabel{double-complex}
Let $\alpha\in\Lambda$ generic, and let $\omega$ be an ordering on
$I^\alpha$. Then the natural sequence of morphisms of complexes
\begin{equation}\elabel{E:complexes-sequence}
0\ra \bfT(\phi_{I(\alpha)}) \lra \bigoplus_{b\in I_\alpha}
\bfT(\phi_{I(\alpha)\cup\{b\}})
 \lra 
    \cdots
\bigoplus_{\begin{smallmatrix}
                   b\in I_\alpha\\
              \end{smallmatrix}}
\bfT(\phi_{I^\alpha\setminus b}) \lra \bfT(\phi_{I^\alpha})
\ra 0
\end{equation}
is acyclic.
\end{lemma}

We can now prove the main result of this paper:

\begin{theorem}\mlabel{main-thm}
Let $\alpha \in \Lambda$ be generic. There is at most one $i\ge 1$
such that $\beta_{i,\alpha}(L)\ne 0$. More precisely, we have for
each $i\ge 1$
\[
\beta_{i,\alpha}(L)=
\begin{cases}
\beta(\bfM_\alpha)  &\text{ if } i = |I_\alpha|- \rank\bfM_\alpha +1 ; \\
0                  &\text{ otherwise}.
\end{cases}
\]
\end{theorem}

\begin{proof}
Let $b\in I_a$ and let $\om$ be an ordering on $I^\a$ so that $b$
is the biggest element of $I_\a$. The two standard spectral
sequences associated with the double complex of Lemma
\ref{double-complex} collapse. According to the first filtration
we get the complex $V(\a,\phi,\om)_\bullet$, see the remarks
preceding Lemma \ref{T-complexes}. According to the second
filtration we get the complex $C(\a)_\bullet$. Thus
$\HH_{i}\bigl(C(\alpha))=\HH_{i}(V(\alpha))$. Combining this with
Lemma~\ref{first-quotient}  and Corollary~\ref{beta} we are done.
\end{proof}

We apply the theorem to monomial ideals. When $J$ is a
monomial ideal and $\a$ is generic we prove that
$\beta_{i,\a}(R/J)\neq 0$ if and only if $\a$ corresponds
to a face of the Scarf complex of $R/J$.

\begin{corollary}\label{monomial-corollary}
Let $J$ be a monomial ideal and let $\alpha\in\Lambda(R/J)$ be a
generic element. If $I_\a\neq I^\a$  then $\forall i\ge 1$,
$\beta_{i,\alpha}(R/J)=0$ . Otherwise
\[
\beta_{i,\alpha}(R/J)=
\begin{cases}
1  &\text{ if } i = |I_\alpha| ; \\
0  &\text{ otherwise}.
\end{cases}
\]
\end{corollary}

\begin{proof}
If $I_\alpha\neq I^{\alpha}$ then $\bfM_\a$ is the empty matroid
and $\beta(\bfM_\alpha)=0$. Otherwise $\beta(\bfM_\alpha)=\rank\bfM_\alpha=1$.
\end{proof}

We finish this section with an example to show that in the general
case it may be  $I_\a\neq I^\a$ and  $\beta_{i,\alpha}(L)\neq 0$.

\begin{example}
 Let $L$ be the module of Example \ref{E:minors-example}. For
 $\a=(3,3,3)$ we have the following data:
$\beta(\bfM_\alpha)=\rank\bfM_\alpha=1$,
 $|I_\a|=3$. Thus $\beta_{3,\alpha}(L)=1$.
\end{example}

\end{document}